\documentclass[oneside,12pt]{amsart}

\usepackage{amsmath,amssymb,amsthm,amsfonts,amstext,amsbsy,amscd,mathrsfs,graphics,graphicx,geometry,color}
\usepackage[colorlinks=true, pdfstartview=FitV, linkcolor=blue, citecolor=blue, urlcolor=blue,pagebackref=false]{hyperref}

\usepackage{amsxtra}
\usepackage[english]{babel}
\usepackage[T1]{fontenc}
\usepackage{latexsym}
\usepackage[latin1]{inputenc}
\newtheorem{theorem}{Theorem}[section]
\newtheorem{lemma}[theorem]{Lemma}

\newtheorem{corollary}[theorem]{Corollary}
\newtheorem{proposition}[theorem]{Proposition}
\newtheorem{example}[theorem]{Example}
\newtheorem{remark}[theorem]{Remark}

\newtheorem{definition}[theorem]{Definition}
\def\bit{\begin{itemize}}
\def\eit{\end{itemize}}
\reversemarginpar   
\def\bc{\begin{center}}
\def\ec{\end{center}}
\def\bthm{\begin{theorem}}
\def\ethm{\end{theorem}}
\def\bcor{\begin{corollary}}
\def\ecor{\end{corollary}}
\def\bprop{\begin{proposition}}
\def\eprop{\end{proposition}}
\def\blem{\begin{lemma}}
\def\elem{\end{lemma}}

\def\brem{\begin{remark}}
\def\erem{\end{remark}}

\def\bdes{\begin{description}}
\def\edes{\end{description}}

\def\beq{\begin{equation}}
\def\eeq{\end{equation}}

\def\ben{\begin{enumerate}}
\def\een{\end{enumerate}}

\def\beqar{\begin{eqnarray}}
\def\eeqar{\end{eqnarray}}
\def\beqarr{\begin{eqnarray*}}
\def\eeqarr{\end{eqnarray*}}

\def\EE{{\mathbb E}}
\def\PP{{\mathbb P}}

\def\cD{\mathcal{D}}

\def\part{\partial}

\def\d#1dt{\frac{d#1}{dt}}    

\begin{document}
\title{On a coupling of solutions to the interface SDE on a star graph}
\maketitle
\begin{center}
\renewcommand{\thefootnote}{(\arabic{footnote})}
  \scshape Hatem Hajri \footnote{Institut VEDECOM, 77 Rue des Chantiers, 78000 Versailles. Email: hatem.hajri@vedecom.fr}
  and Marc Arnaudon \footnote{Institut de Mathématiques de Bordeaux UMR 5251, 351, Cours de la Libération - F33405 TALENCE. Email:
marc.arnaudon@math.u-bordeaux.fr}\setcounter{footnote}{0}
\end{center}

\begin{abstract}

Inspired by Tsirelson proof of the non Brownian character of Walsh Brownian motion filtration on three or more rays, we prove some results on a particular coupling of solutions to the interface SDE on a star graph, recently introduced in \cite{MR2905755}. This coupling consists in two solutions which are independent given the driving Brownian motion. As a consequence, we deduce that if the star graph contains $3$ or more rays, the argument of the solution at a fixed time is independent of the driving Brownian motion.
\end{abstract}

\section{Introduction and main results}

A filtration $(\mathcal F_t)_t$ has the Brownian representation property (BRP) if there exists a Brownian motion $B$ such that every $(\mathcal F_t)_t$-martingale is a stochastic integral of $B$. In 1979 Yor posed the reverse problem, i.e whether a filtration having the BRP is necessarily Brownian \cite{MR155}. At the end of his paper \cite{MR509476}, Walsh suggested the study of a Markov process with state space
$$G=\bigcup_{j=1}^{N} E_j;\ E_j=\{r e^{i\theta_j} : r\ge 0\}$$
where $\theta_j$ are given angles. This process, called since then Walsh Brownian motion (WBM), behaves like a standard Brownian motion on
each ray; and at $0$ it makes excursions with probability $p_j$ on $E_j\setminus\{0\}$. Later on, a detailed study of WBM was given in \cite{MR1022917}. In particular, it was shown that WBM is a strong Markov process with Feller semigroup and that the natural filtration $(\mathcal F^Z_t)_t$ of a WBM $Z$ has the BRP with respect to the Brownian motion $B$ given by the martingale part of $|Z|$, the geodesic distance between $Z$ and $0$.

After nearly two decades a negative answer to Yor's question was finally given by Tsirelson \cite{MR1487755}. The result proved by Tsirelson is the following

\begin{theorem}\label{tr}
If $(\mathcal G_t)_t$ is a Brownian filtration, i.e a filtration generated by a finite or infinite family of independent standard Brownian motions, there does not exist any $(\mathcal G_t)_t$-WBM ($(\mathcal G_t)_t$-Markov process with semigroup $P$, the Feller semigroup of WBM) on a star graph with three or more rays.

\end{theorem}

To prove Theorem \ref{tr}, Tsirelson performs a beautiful reasoning by contradiction. Suppose there exists a Brownian motion $B$ such $Z=F(B)$ is a WBM with $N\ge 3$ rays. Let $Z^r=F(B^r)$ where $B^r=rB+\sqrt{1-r^2} B'$ with $B'$ an independent copy of $B$. Then, it is shown that $\mathbb E[d(Z^r_t,Z_t)]$ converges to $0$. However, Tsirelson is able to prove that $\mathbb E[d(Z^r_t,Z_t)]>c>0$ with $c$ not depending on $r$.

In the present paper we are interested in a simple stochastic differential equation on $G$ whose solutions are WBMs. This SDE is the interface SDE introduced in \cite{MR2905755} and driven by an $N$ dimensional Brownian motion $W=(W^1,\cdots,W^N)$. While moving inside $E_i$, a solution to this equation follows $W^i$ so that the origin can be seen as an interface at the intersection of the half lines. For $N=2$, the interface SDE is identified with 
\begin{equation}\label{sde}
dX_t=1_{\{X_t>0\}} dW^1_t + 1_{\{X_t\le 0\}} dW^2_t
\end{equation}
Equation (\ref{sde}) has a unique strong solution \cite{MR18,MR000}. Not knowing Theorem \ref{tr}, one could have the intuition, that similarly to $N=2$, solutions are also strong ones for $N\ge 3$. The Theorem implies this cannot be the case. 

The main result proved in \cite{MR2905755} was the existence of a stochastic flow of mappings, unique in law and a Wiener stochastic flow \cite{MR2060298} which solve the interface SDE. The problem of finding the flows of kernels which ``interpolate''  between these two particular flows was left open in \cite{MR2905755}. The answer to this question needs a complete understanding of weak solutions of this equation.

The purpose of the present paper is to establish new results on weak solutions of the interface SDE in the case $N\ge 3$. These results are very different from the case $N=2$. Our proofs are largely inspired by Tsirelson proof of Theorem \ref{tr}.

\subsection{Notations}
\leavevmode\par
\vspace{0.2cm}

\noindent This paragraph contains the main notations and definitions which will be used throughout the paper.  

Let $(G,d)$ be a metric star graph with a finite set of rays $(E_i)_{1\le i\le N}$ and origin denoted by $0$. This means that $(G,d)$ is a metric space, $E_i\cap E_j=\{0\}$ for all $i\ne j$  and for each $i$, there is an isometry $e_i:[0,\infty[\to E_i$. We assume $d$ is the geodesic distance on $G$ in the sense that $d(x,y)=d(x,0)+d(0,y)$ if $x$ and $y$ do not belong to the same $E_i$. 

For any subset $A$ of $G$, we will use the notation $A^{\ast}$ for $A\setminus\{0\}$. Also, we define the function $\varepsilon:G^{\ast}\rightarrow\{1,\cdots,N\}$ by $\varepsilon(x)=i$ if $x\in E_i^{\ast}$.  

Let $C^2_b(G^*)$ denote the set of all continuous functions $f:G\to\mathbb{R}$ such that for all $i\in [1,N]$, $f\circ e_i$ is $C^2$ on $]0,\infty[$ with bounded first and second derivatives both with finite limits at $0+$.
For $x=e_i(r)\in G^*$, set $f'(x)=(f\circ e_i)'(r)$ and $f''(x)=(f\circ e_i)''(r)$.

Let $p_1,\cdots,p_N\in (0,1)$ such that $\sum_{i=1}^{N} p_i=1$ and define
$$\cD=\left\{f\in C^2_b(G^*) : \sum_{i=1}^{N} p_i(f\circ e_i)'(0+)=0\right\}.$$

For $f\in C^2_b(G^*)$, we will take the convention $f'(0)=\sum_{i=1}^{N} p_i(f\circ e_i)'(0+)$ and $f''(0)=\sum_{i=1}^{N} p_i(f\circ e_i)''(0+)$ so that $\cD$ can be written as $\cD=\left\{f\in C^2_b(G^*) : f'(0)=0\right\}$. We are now in position to recall the following

\begin{definition}\label{def}
A solution of the interface SDE $(I)$ on $G$ with initial condition $X_0=x$ is a
pair of processes $(X,W)$ defined on a filtered probability space
$(\Omega,\mathcal A,(\mathcal F_t)_t,\mathbb P)$ such that
\begin{itemize}
\item[(i)] $W=(W^1,\dots,W^N)$ is a standard $(\mathcal
F_t)$-Brownian motion in $\mathbb R^N$;
\item[(ii)] $X$ is an $(\mathcal F_t)$-adapted continuous process on $G$;
\item[(iii)] For all $f\in \mathcal{D}$,
\begin{equation}\label{kdd00}
f(X_t)=f(x)+\sum_{i=1}^N \int_{0}^{t}f'(X_s)1_{\{X_s\in
E_i\}}dW^i_s+\frac{1}{2}\int_{0}^{t} f''(X_s)ds
\end{equation}
\end{itemize}

\end{definition}
To emphasize on the filtration $(\mathcal F_t)_t$, we will sometimes say $(X,W)$ is an $(\mathcal F_t)_t$-solution. It has been proved in \cite{MR2905755} (Theorem 2.3) that for all $x\in G$, $(I)$ admits a solution $(X,W)$ with $X_0=x$, the law of $(X,W)$ is unique and $X$ is an $(\mathcal F_t)$- WBM on $G$. We will denote by $Q_x$ the law of a solution $(X,W)$ with $X_0=x$.

Tsirelson theorem \ref{tr} combined with Theorem 2.3 in \cite{MR2905755} show that $X$ is $\sigma(W)$-measurable if and only if $N\le 2$.

Let us give an intuitive description of solutions to the previous equation. Given a WBM $X$ started from $x$, we will denote from now on by $B^X$ the martingale part of $|X|-|x|$. Freidlin-Sheu formula \cite{MR1743769} says that for all $f\in C^2_b(G^*)$

\begin{equation}\label{fs}
f(X_t)=f(x)+\int_{0}^{t}f'(X_s) dB^X_s+\frac{1}{2}\int_{0}^{t} f''(X_s)ds + \sum_{i=1}^{N} p_i(f\circ e_i)'(0+) L_t(|X)|
\end{equation}

with $L_t(|X|)$ denoting the local time of $|X|$. Comparing (\ref{kdd00}) with (\ref{fs}), one gets
\begin{equation}\label{dedi}
B^X_t=\sum_{i=1}^{N} \int_{0}^{t} 1_{\{X_s\in E_i\}} dW^i_s
\end{equation}
Thus, while it moves inside $E_i$, $X$ follows the Brownian motion $W^i$ which shows that (\ref{kdd00}) extends (\ref{sde}) in a natural way. 
\vspace{0.1cm}

Let us now introduce the following

\begin{definition}
We say that $(X,Y,W)$ is a coupling of solutions to $(I)$ if $(X,W)$ and $(Y,W)$ satisfy Definition \ref{def} on the same filtered probability space $(\Omega,\mathcal A,(\mathcal F_t)_t,\mathbb P)$.
\end{definition}
A trivial coupling of solutions to $(I)$ is given by $(X,X,W)$ where $(X,W)$ solves $(I)$. This is also the law unique coupling of solutions to $(I)$ if $N\le 2$ as $\sigma(X)\subset \sigma(W)$ in this case. Let us now introduce another interesting coupling.
\begin{definition}

A coupling $(X,Y,W)$ of solutions to $(I)$ is called the Wiener coupling if $X$ and $Y$ are independent given $W$. 

\end{definition}

The existence of the Wiener coupling is easy to check. For this note there exists a law unique triplet $(X,Y,W)$ such that $(X,W)$ and $(Y,W)$ are distributed respectively as $Q_x$ and $Q_y$ and moreover $X$ and $Y$ are independent given $W$. It remains to check that $W$ is a standard $(\mathcal F_t)$-Brownian motion in $\mathbb R^N$ where $\mathcal F_t=\sigma(X_u, Y_u, W_u, u\le t)$. This holds from the conditional independence between $X$ and $Y$ given $W$ and the fact that $W$ is a Brownian motion with respect to the natural filtrations of $(X,W)$ and $(Y,W)$. The reason for choosing the name Wiener for this coupling will be justified in Section \ref{ff} in connection with stochastic flows.

%
%

\subsection{Main results}
\leavevmode\par
\vspace{0.2cm}
Given a WBM $X$ on $G$, we define the process $\overline{X}$ by 
$$\overline{X}_t = 1_{\{X_t\ne 0\}} \sum_{i=1}^{N} 1_{\{\varepsilon(X_t)=i\}}\times e_i\left(\frac{e_i^{-1}(X_t)}{Np_i}\right)$$
Note that $\overline{X}=X$ if $p_i=\frac{1}{N}$ for all $1\le i\le N$. Following the terminology used in \cite{MR1655299}, the process $\overline{X}$ is a spidermartingale (``martingale-araign\'ee''). In fact, for all $1\le i\le N$, define 
\begin{equation}\label{d}
\overline{X}^i_t=|\overline{X}_t|\ \ \text{if}\ \ \overline{X}_t\in E_i\ \text{and}\ \overline{X}^i_t=0\ \ \text{if not}
\end{equation}
Note that $\overline{X}^i_t=f^i(X_t)$, where $f^i(x)=\frac{|x|}{N p_i} 1_{\{x\in E_i\}}$. Applying Freidlin-Sheu formula (\ref{fs}) for $X$ and the function $f^i$ shows that
\begin{equation}\label{Maaa}
\overline{X}^i_t=\frac{1}{N p_i}\int_{0}^{t} 1_{\{X_s\in E_i\}} dB^X_s + \frac{1}{N} L_t(|X|)
\end{equation}
In particular, 
$\overline{X}^i_t-\overline{X}^j_t$ is a martingale for all $i,j\in[1,N]$. Proposition 5 in \cite{MR1655299} shows that $\overline{X}$ is a spidermartingale. 

Our main result in this paper is the following.
\begin{theorem}\label{main}
Assume $N\ge 3$. Let $(X,Y,W)$ be the Wiener coupling of solutions to $(I)$ with $X_0=Y_0=0$. Then 
\begin{itemize}
\item[(i)] $d(\overline{X}_t,\overline{Y}_t)-\frac{N-2}{N}(|\overline{X}_t|+|\overline{Y}_t|)$ is a martingale. In particular, 
$$\EE[d(\overline{X}_t,\overline{Y}_t)]=2\frac{N-2}{N} \sqrt{\frac{2t}{\pi}}.$$
\item[(ii)] Call $g^X_t$ and $g^Y_t$ the last zeroes before $t$ of $X$ and $Y$, then for all $t>0$, $\mathbb P(g^X_t = g^Y_t)=0$ and $\mathbb P(X_t = Y_t)=0$.
\item[(iii)] $\varepsilon(X_t)$ and $\varepsilon(Y_t)$ are independent for all $t>0$.
\end{itemize}
\end{theorem} 

Another important fact about $(X,Y,W)$ proved in \cite{MR2905755}, also true for $N=2$, says that $(X,Y,W)$ is a strong Markov process associated with a Feller semigroup. This result will be sketched in Section \ref{ff} below. 

 The claim (ii) says that common zeros of $X$ and $Y$ are rare. It has been proved in \cite{MR2905755}, that couplings $(X,Y)$ to $(I)$ have the same law before $T=\inf\{t\ge 0 : X_t=Y_t\}$ and that $T<\infty$ with probability one. The strong Markov property shows then that the set of common zeros of $X$ and $Y$ is infinite.

\textbf{The case $N=2$.} Point (i) in Theorem \ref{main} is also true for $N=2$ since $X=Y$ in this case \cite{MR2905755}. This can also be deduced from the proofs below. In fact, Proposition \ref{time} claims that $\Lambda_t$ defined as the local time of the semimartingale $d(\overline{X}_t,\overline{Y}_t)$ is zero for all $N\ge 2$. By the usual Tanaka formula (see also Proposition \ref{reminder}),
$$d(\overline{X}_t,\overline{Y}_t)=M_t+\frac{1}{2} \Lambda_t $$
where $M$ is a martingale. Taking the expectation shows that $\overline{X}=\overline{Y}$ and so $X=Y$. The same reasoning applies to any coupling $(X,Y)$ and in particular pathwise uniqueness holds for (\ref{kdd00}) in the case $N=2$. Since weak uniqueness is also satisfied, this yields the strong solvability of (\ref{kdd00}) when $N=2$ (or equivalently (\ref{sde})).  
\vspace{0.2cm}

Theorem \ref{main} yields the following important
\begin{corollary}\label{nice}
Assume $N\ge 3$. Let $(X,W)$ be a solution of $(I)$ with $X_0=0$. Then for each $t>0$, $\varepsilon(X_t)$ is independent of $W$.
\end{corollary} 
This corollary seems to us quite remarkable. In fact, admitting Tsirelson theorem \ref{tr} and using (\ref{dedi}), it can be deduced that $\varepsilon(X_t)$ is not $\sigma(W)$-measurable (actually neither $\varepsilon(X_t)$ nor $|X_t|$ are $\sigma(W)$-measurable). However, Corollary (\ref{nice}) gives a much stronger result than this non-measurability. Comparing this with the case $N=2$, in which $\epsilon(X_t)$ is $\sigma(W)$-measurable, shows that stochastic differential equations on star graphs with $N\ge 3$ rays involve interesting ``phase transitions''.

Corollary \ref{nice} is easy to deduce from Theorem \ref{main}. For this, define $C_t=\mathbb P(\varepsilon(X_t)=i|W)$. Since $X$ and $Y$ are independent given $W$ and $(X,W)$, $(Y,W)$ have the same law, 
$$\mathbb P\left(\varepsilon(X_t)=i\right)^2=\mathbb P(\varepsilon(X_t)=i, \varepsilon(Y_t)=i)=\EE[C_t^2]$$
Thus $\EE[C_t]=\EE[C_t^2]^{\frac{1}{2}}$ and so there exists a constant $c_t$ such that $C_t=c_t$ a.s. Taking the expectation shows that $c_t=p_i$. 

Let us now explain our arguments to prove Theorem \ref{main}. 

In Section \ref{s}, we prove that for any coupling $(X,Y,W)$ of solutions to $(I)$ such that $X_0=Y_0=0$, we have $L_t(D)=0$ where $D_t=d(\overline{X}_t,\overline{Y}_t)$. 

Next, inspired by Tsirelson arguments \cite{MR1487755}, we consider a perturbation $W^r=rW+\sqrt{1-r^2} \hat{W}, r<1$ of $W$, $\hat{W}$ is an independent copy of $W$, and $X^r, Y^r$ such that 
\begin{itemize}
\item $(X^r,W^r)$ and $(Y^r,W)$ are solutions to $(I)$.
\item $X^r$ and $Y^r$ are independent given $(W,\hat{W})$.

\end{itemize}
The coupling $(X^r,Y^r)$ satisfies
\begin{equation}\label{crucial}
\frac{d}{dt}\langle |X^r|,|Y^r|\rangle_t\le r<1
\end{equation}
A crucial result proved in \cite{MR1487755} (see also \cite{MR1655299,MR1655300}) which will be used below says that, since (\ref{crucial}) holds, $L_t(|X^r|)$ and $L_t(|Y^r|)$ have rare common points of increase (see (ii) in Proposition \ref{pop} for more precision). The process $(X^r,Y^r,W)$ is shown to converge in law as $r\rightarrow1$ to the Wiener coupling $(X,Y,W)$ described above. The passage to the limit $r\rightarrow1$ allows to deduce the properties mentioned in Theorem \ref{main}. 

Section \ref{ff} is a complement based on stochastic flows to the previous results. We consider the Wiener stochastic flow of kernels $K$ constructed in \cite{MR2905755} which is a strong solution to the flows of kernels version of (\ref{kdd00}) driven by a real white noise $(W_{s,t})_{s\le t}$. The Wiener coupling $(X,Y,W)$ is shown to be the strong Markov process associated to a Feller semigroup $Q$ obtained from $K$. 
%

\section{Proofs}

\subsection{The local time of the distance}\label{s}
\leavevmode\par
\vspace{0.2cm}
The subject of this paragraph is to prove the following result which, in the case $N=2$, is proved in \cite{MR2905755}.
\begin{proposition}\label{time} Assume $N\ge 2$. Let $(X,Y,W)$ be a coupling of two solutions to $(I)$ with $X_0=Y_0=0$ and let ${D}_t=d(\overline{X}_t,\overline{Y}_t)$. Then $D$ is a semimartingale with $L_t({D})=0$.

\end{proposition}
\begin{proof}
The fact that $D$ is a semimartingale is shown in \cite{MR1487755} (see \cite{MR1655299} and Proposition \ref{reminder} below for more details). We follow the proof of Proposition 4.5 in \cite{MR2905755} and first prove that a.s.

\begin{equation}\label{ddd}
\int_{]0,+\infty]} L^a_t(D)\frac{da}{a}<\infty
\end{equation}

where $L^a_t(D)$ is the local time of $D$ at level $a$ and time $t$. 
Recall that by the occupation formula

$$\int_{]0,+\infty]} L^a_t(D)\frac{da}{a}=\int_{0}^{t}1_{\{D_s>0\}} \frac{d\langle D\rangle_s}{D_s}$$
By (\ref{Maaa}), 
\begin{eqnarray}
|\overline{X}_t|&=&\sum_{i=1}^{N} \overline{X}^i_t = M^1_t + L_t(|X|)\nonumber\\
|\overline{Y}_t|&=&\sum_{i=1}^{N} \overline{Y}^i_t = M^2_t + L_t(|Y|)\nonumber\
\end{eqnarray}
with
$$M^1_t = \sum_{i=1}^{N} \frac{1}{N p_i}\int_{0}^{t} 1_{\{X_s\in E_i\}} dB^X_s,\ M^2_t=\sum_{i=1}^{N} \frac{1}{N p_i}\int_{0}^{t} 1_{\{Y_s\in E_i\}} dB^Y_s$$
In particular, 
$$\langle M^1\rangle_t=\sum_{i=1}^{N} \frac{1}{(N p_i)^2}\int_{0}^{t} 1_{\{X_s\in E_i\}} ds,\ \langle M^2\rangle_t=\sum_{i=1}^{N} \frac{1}{(N p_i)^2}\int_{0}^{t} 1_{\{Y_s\in E_i\}} ds$$
and
$$\langle M^1,M^2\rangle_t = \sum_{i=1}^{N} \frac{1}{(N p_i)^2}\int_{0}^{t} 1_{\{X_s\in E_i,\ Y_s\in E_i\}} ds.$$
Proposition 7 \cite{MR1655299} tells us that 
\begin{eqnarray}
D_t&-&\int_{0}^{t} 1_{\{\varepsilon(X_s)\ne\varepsilon(Y_s)\}} (dM^1_s+dM^2_s) \nonumber\\
&-& \int_{0}^{t} 1_{\{\varepsilon(X_s)=\varepsilon(Y_s)\}} \text{sgn}(M^1_s - M^2_s) (dM^1_s-dM^2_s)\nonumber\
\end{eqnarray}

is a continuous increasing process. Consequently,
$$d\langle D\rangle_s = \sum_{i=1}^{N} \frac{1}{(N p_i)^2} 1_{\{\varepsilon(X_s)\ne\varepsilon(Y_s)\}} (1_{\{X_s\in E_i\}}+1_{\{Y_s\in E_i\}}) ds\le C 1_{\{\varepsilon(X_s)\ne\varepsilon(Y_s)\}} ds$$
where $C$ is a positive constant. Note there exists $C'>0$ such that $D_s\ge C'(|X_s|+|Y_s|)$ for all $s$ such that $\varepsilon(X_s)\ne\varepsilon(Y_s)$. Thus, to get (\ref{ddd}), it is sufficient to prove 
$$\int_{0}^{t}1_{\{X_s\ne 0,Y_s\ne 0\}}1_{\{\epsilon(X_s)\ne \epsilon(Y_s)\}} \frac{ds}{|X|_s+|Y_s|}<\infty$$
Let us prove for instance that
$$(1)=\int_{0}^{t}\frac{1}{|X_s| + |Y_s|}1_{\{X_s\in E_1^*, Y_s\notin E_1\}} ds<\infty$$
Define $f(z)=|z|$ if $z\in E_1$ and $f(z)=-|z|$ if not and set $x_t=f(X_t), y_t=f(Y_t)$. Clearly
$$\frac{1}{|X_s| + |Y_s|}1_{\{X_s\in E_1^*, Y_s\notin E_1\}}\,=\, \frac{1}{2}\frac{|\text{sgn}(x_s) - \text{sgn}(y_s)|}{|x_s - y_s|}1_{\{y_s< 0 < x_s\}}.$$
As in \cite{MR2905755}, let $(f_n)_n\subset C^1(\mathbb R)$ such that $f_n\rightarrow\text{sgn}$ pointwise and $(f_n)_n$ is
uniformly bounded in total variation. Defining $z^u_s=(1-u) x_s + u y_s$, we have by Fatou's Lemma
\begin{eqnarray}
(1)&\le& \liminf_n\int_{0}^{t}1_{\{y_s<0<x_s\}}\frac{|f_n(x_s)-f_n(y_s)|}{|x_s - y_s|} \frac{ds}{2}\nonumber\\
&\le &\liminf_n\int_{0}^{t} 1_{\{y_s<0<x_s\}} \int_{0}^{1}\big|f'_n(z^u_s)\big| du \frac{ds}{2}\nonumber
\end{eqnarray}
Writing Freidlin-Sheu formula for the function $f$ applied to $X$ and $Y$ shows that on $\{y_s<0<x_s\}$,
$$\frac{d}{ds}\langle z^u\rangle_s=u^2+(1-u)^2 \ge \frac{1}{2}$$
Thus
\begin{eqnarray*}(1) &\le&  \liminf_n \int_0^1 \int_{0}^{t} 1_{\{y_s<0<x_s\}} \big|f'_n(z^u_s)\big| d\langle z^u\rangle_s du\\
&\le& \liminf_n\int_{0}^{1}\int_{\mathbb R} \big|f'_n(a)\big| L^a_t(z^u)dadu
\end{eqnarray*}
So a sufficient condition for (1) to be finite is
$$\sup_{a\in\mathbb R,u\in[0,1]}\EE\big[L^a_t(z^u)\big]<\infty$$
By Tanaka's formula
\begin{eqnarray*}
\EE\big[L^a_t(z^u)\big]&=&
\EE\big[\big|z^u_t-a\big|\big]-\EE\big[\big|z^u_0-a\big|\big]-\EE\bigg[\int_{0}^{t}\text{sgn}(z^u_s-a)dz^u_s\bigg]\\
&\le& \EE[\big|z^u_t-z^u_0\big|] -\EE\bigg[\int_{0}^{t}\text{sgn}(z^u_s-a)dz^u_s\bigg]
\end{eqnarray*}
Since $x$ and $y$ are two skew Brownian motions, it is easily seen that $\sup_{u\in[0,1]} \EE[\big|z^u_t-z^u_0\big|]<\infty$. The same argument shows that $$\EE\bigg[\int_{0}^{t}\text{sgn}(z^u_s-a)dz^u_s\bigg]$$
is uniformly bounded with respect to $(u,a)$ and consequently (1) is finite. Finally $\int_{]0,+\infty]} L^a_t(D)\frac{da}{a}$ is finite a.s. Since $\lim_{a\downarrow 0} L^a(D)=L^0(D)$, we deduce $L^0_t(D)=0$. 
\end{proof}

\subsection{Proof of Theorem \ref{main}}
\leavevmode\par
\vspace{0.2cm}

This section gives the proof of our main result. First, we define the perturbation of the Wiener coupling as described in the introduction and then perform a passage to the limit. 
\begin{lemma}\label{lemm}
For all $r\in[0,1]$, there exists a law unique process $(X,W,\hat{W})$ such that, denoting $\mathcal F_t=\sigma(X_u,W_u,\hat{W}_u, u\le t)$,

\begin{itemize}
\item $W$ and $\hat{W}$ are two independent $(\mathcal F_t)_t$-Brownian motions in $\mathbb R^N$. 
\item $(X,W^r)$ is an $(\mathcal F_t)_t$-solution to $(I)$ with $X_0=0$ and where $W^r=r W +\sqrt{1-r^2} \hat{W}$.

\end{itemize}

\end{lemma}

\begin{proof}
The proof of this lemma is similar to that of Theorem 2.3 in \cite{MR2905755}. For the existence part, take independent processes $X,V^1,\cdots,V^N,\cdots,V^{2N}$
where $X$ is a WBM started from $0$ and each $V^i$ is a standard Brownian motion. Denote by $(\mathcal G_t)_t$ the natural filtration of $(X,V^1,\cdots,V^N,\cdots,V^{2N})$ and for $1\le i\le N$, define
$$d\Gamma^i_t=1_{\{X_t\in E_i\}} dB^X_t + 1_{\{X_t\notin E_i\}} dV^i_t,\ dW^i_t=r d\Gamma^i_t+\sqrt{1-r^2}dV^{i+N}_t$$
and
$$d\hat{W}^i_t=\sqrt{1-r^2} d\Gamma^i_t - r dV^{i+N}_t.$$
Then $W=(W^1,\cdots,W^N), \hat{W}=(\hat{W}^1,\cdots,\hat{W}^N)$ are two independent $(\mathcal G_t)_t$-Brownian motions in $\mathbb R^N$, $(X,\Gamma^1,\cdots,\Gamma^N)$ is a $(\mathcal G_t)_t$-solution to $(I)$ and since $d\Gamma^i_t=r dW^i_t + \sqrt{1-r^2} d\hat{W}^i_t$, existence holds. 

Now let $(X,W,\hat{W})$ and $(\mathcal F_t)_t$ be as in the lemma. Introduce a Brownian motion $B$ independent of $(X,W,\hat{W})$ and define $\mathcal G_t=\mathcal F_t\vee \sigma(B_u, u\le t)$. Write $\hat{W}=(\hat{W}^1,\cdots,\hat{W}^N)$ and for $1\le i\le N$, define $$d\Gamma^i_t=r dW^i_t + \sqrt{1-r^2} d\hat{W}^i_t,\ \ \ dV^{i+N}_t=\sqrt{1-r^2} dW^i_t - r d\hat{W}^i_t$$
and
$$ dV^i_t=1_{\{X_t\notin E_i\}} d\Gamma^i_t + 1_{\{X_t\in E_i\}} dB_t.$$ 
Note that $V^1,\cdots,V^N,\cdots,V^{2N}$ are independent $(\mathcal G_t)_t$-Brownian motions. Using $1_{\{X_t\in E_i\}}dB^X_t=1_{\{X_t\in E_i\}} d\Gamma^i_t$, simple calculations show that $(V^1,\cdots,V^N,\cdots,V^{2N})$ is independent of $B^X$. Since $X$ is a $(\mathcal G_t)_t$-WBM, Lemma 4.3 in \cite{MR2905755} claims that $X,V^1,\cdots,V^N,\cdots,V^{2N}$ are independent. Finally $(X^+,W^+,\hat{W}^{+})$ constructed from $X,V^1,\cdots,V^N,\cdots,V^{2N}$ as in the existence part coincides with $(X,W,\hat{W})$. This finishes the proof. 
  
  
\end{proof}
An immediate consequence of the previous lemma is the following

\begin{lemma}
For all $r\in[0,1]$, there exists a law unique process $(X,Y,W,\hat{W})$ such that, denoting $\mathcal F_t=\sigma(X_u,Y_u,W_u,\hat{W}_u, u\le t)$,

\begin{itemize}
\item $W$ and $\hat{W}$ are two independent $(\mathcal F_t)_t$-Brownian motions in $\mathbb R^N$.
\item $(X,W^r)$ and $(Y,W)$ are two $(\mathcal F_t)_t$-solutions to $(I)$ with $X_0=Y_0=0$ and where $W^r=r W +\sqrt{1-r^2} \hat{W}$.
\item $X$ and $Y$ are independent given $(W,\hat{W})$.

\end{itemize}
\end{lemma}
The proof of this lemma is similar to the existence and law uniqueness of the Wiener coupling and is left as an exercise. 

In the sequel, we will denote $(X,Y,W,\hat{W})$ by $(X^r,Y^r,W,\hat{W})$ and use the notation $W^r$ to denote $r W +\sqrt{1-r^2} \hat{W}$.

\begin{proposition}\label{pop}
The following assertions hold  
\begin{itemize}
\item [(i)] $d\langle B^{X^r},B^{Y^r}\rangle_t= r 1_{\{\varepsilon(X^r_t)=\varepsilon(Y^r_t)\}} dt$.
\item[(ii)] $\int_0^t 1_{\{Y^r_s\ne 0\}} dL_s(|X^r|)=L_t(|X^r|)$ and $\int_0^t 1_{\{X^r_s\ne 0\}} dL_s(|Y^r|)=L_t(|Y^r|)$.

\end{itemize}

\end{proposition}
\begin{proof} 
Write $W^r=(W^{r,1},\cdots,W^{r,N})$. By the previous lemma $$dB^{X^r}_t=\sum_{i=1}^{N} 1_{\{X^r_t\in E_i\}} dW^{r,i}_t \ \text{and}\ dB^{Y^r}_t=\sum_{i=1}^{N} 1_{\{Y^r_t\in E_i\}} d W^{i}_t$$ which yields (i). (ii) is Lemma 4.12 in \cite{MR1487755} (see also \cite{MR1655299,MR1655300}).  
\end{proof}

The next lemma establishes the convergence in law of $(X^r,Y^r,W)$ to the Wiener coupling $(X,Y,W)$. 

\begin{lemma}\label{lem}
 As $r\to 1$, $(X^r,Y^r,W)$ converges in law to $(X,Y,W)$, the Wiener coupling of solutions to $(I)$ with $X_0=Y_0=0$. 
\end{lemma}
\begin{proof} Let $(r_n)_n$ be a sequence in $[0,1]$ such that $\lim_{n\to\infty}r_n=1$. For any $p\ge 1$, $(f_i,g_i,h_i)_{1\le i\le p}$ bounded, $(t_i)_{1\le i\le p}$

$$\EE\left[\prod_{i=1}^{p} f_i(X^{r_n}_{t_i}) g_i(Y^{r_n}_{t_i}) h_i(W_{t_i})\right]=\EE\left[\prod_{i=1}^{p} \EE[f_i(X^{r_n}_{t_i})|W,\hat{W}] \EE[g_i(Y^{r_n}_{t_i})|W,\hat{W}] h_i(W_{t_i})\right]$$
Note that $(Y^{r_n},W)$ is independent of $\hat{W}$. This can be deduced from the uniqueness part in Lemma \ref{lemm} by taking $r=1$. Consequently $\EE[g_i(Y^{r_n}_{t_i})|W,\hat{W}]=\EE[g_i(Y^{r_n}_{t_i})|W]$ a.s.
Now $\sigma(W,\hat{W})=\sigma(W^{r_n},\overline{W}^{r_n})$ where $\overline{W}^{r_n}$ is the independent complement to $W^{r_n}$ given by
$$\overline{W}^{r_n}=\sqrt{1-r_n^2} W - r_n \hat{W}.$$
By the proof of Lemma \ref{lemm}, $(X^{r_n},W^{r_n},\overline{W}^{r_n})$ and $(Y^{r_n},W,\hat{W})$ have the same law. Consequently $(X^{r_n},W^{r_n})$ is also independent of $\overline{W}^{r_n}$ and so $$\EE[f_i(X^{r_n}_{t_i})|W,\hat{W}]=\EE[f_i(X^{r_n}_{t_i})|W^{r_n},\overline{W}^{r_n}]=\EE[f_i(X^{r_n}_{t_i})|W^{r_n}]. $$
Slutsky lemma (see Theorem 1 in \cite{MR1655299}) shows that for all $f:G\rightarrow\mathbb R$ measurable bounded and $t>0$, as $n\rightarrow\infty$, 
$$\mathbb E[f(X^{r_n}_{t})|W^{r_n}]\longrightarrow Q_tf(W)$$
in probability where $Q_tf(W)=\int f(y) Q_t(W,dy)$ and $Q_t(W,dy)$ is a regular conditional expectation of $X_t$ given $W$. Finally  

\begin{eqnarray}
\lim_n \EE\left[\prod_{i=1}^{p} f_i(X^{r_n}_{t_i}) g_i(Y^{r_n}_{t_i}) h_i(W_{t_i})\right]&=&\EE\left[\prod_{i=1}^{p} Q_{t_i}f_i(W) Q_{t_i}g_i(W) h_i(W_{t_i})\right]\nonumber\\
&=& \EE\left[\prod_{i=1}^{p} f_i(X_{t_i}) g_i(Y_{t_i}) h_i(W_{t_i})\right]\nonumber\
\end{eqnarray}
and the lemma is proved. 

\end{proof}
Let us now recall Proposition 7 in \cite{MR1655299}. 
\begin{proposition}\label{reminder}
Let $Z^1$ and $Z^2$ be two WBMs with respect to the same filtration such that $Z^1_0=Z^2_0=0$. Denote by $\Lambda$ the local time of $D_t=d(\overline{Z^1_t},\overline{Z^2_t})$. Then 
$$D_t =M_t + \frac{1}{2} \Lambda_t + (N-2) \left(\int_0^t 1_{\{\overline{Z^1_s}\ne 0\}} dL^2_s + \int_0^t 1_{\{\overline{Z^2_s}\ne 0\}} dL^1_s\right)$$ 
with $M$ a martingale, $M_0=0$ and $L^1, L^2$ are (see Proposition 5 in \cite{MR1655299}) the bounded variation parts of $\overline{X}^i_t$ (defined by (\ref{d})) and $\overline{Y}^i_t$. 
\end{proposition}
Note that $L^1_t=\frac{1}{N} L_t(|Z^1|)$ and $L^2_t=\frac{1}{N} L_t(|Z^2|)$ by (\ref{Maaa}).

Applying the previous proposition to $(Z^1,Z^2)=(X^r,Y^r)$ and using Proposition \ref{pop} (ii), we get

$$d(\overline{X^r_t},\overline{Y^r_t}) =M^r_t + \frac{1}{2} \Lambda^r_t + \frac{(N-2)}{N} \left(L_t(|X^r|) + L_t(|Y^r|)\right)$$ 
with $M^r$ a martingale and $\Lambda^r$ the local time of $d(\overline{X^r_t},\overline{Y^r_t})$. In particular, 
\begin{equation}\label{yes}
\EE[d(\overline{X^r_t},\overline{Y^r_t})]\ge 2\frac{(N-2)}{N} \EE[R_t]
\end{equation}
with $R$ a reflected Brownian motion started from $0$. 

Proposition \ref{reminder} applied to the Wiener coupling $(Z^1,Z^2)=(X,Y)$ and the result of Section \ref{s} show that
\begin{equation}\label{mp}
d(\overline{X_t},\overline{Y_t}) =M_t + \frac{(N-2)}{N} \left(\int_0^t 1_{\{X_s\ne 0\}} dL_s(|Y|) + \int_0^t 1_{\{Y_s\ne 0\}} dL_s(|X|)\right)
\end{equation} 
with $M$ a martingale. By the Balayage formula (see \cite{MR2200733} on page 111 or the proof of Proposition 8 in \cite{MR1655299}) and the fact that $L_t(D)=0$, 
\begin{equation}\label{opo}
d(\overline{X_t},\overline{Y_t}) = \text{Martingale} + \frac{N-2}{N}\left(1_{\{\overline{X}_{g^2}\ne 0\}}|\overline{Y_t}| + 1_{\{\overline{Y}_{g^1}\ne 0\}}|\overline{X_t}|\right)
\end{equation}
where $g^1:=g^{\overline{X}}_t$ and $g^2:=g^{\overline{Y}}_t$. Admit for a moment that $\EE[d(\overline{X^r_t},\overline{Y^r_t})]$ converges to $\EE[d(\overline{X_t},\overline{Y_t})]$. It comes from (\ref{yes}), (\ref{opo}), $(X,Y)$ has the same law as $(Y,X)$, that
$$2\frac{N-2}{N} \EE\left[1_{\{\overline{X}_{g^2}\ne 0\}}|\overline{Y_t}|\right]\ge 2\frac{(N-2)}{N} \EE[R_t]$$ 
Consequently 
$$\EE\left[|\overline{Y_t}|\right]\ge \EE\left[1_{\{\overline{X}_{g^2}\ne 0\}}|\overline{Y_t}|\right]\ge \EE[R_t]$$ 
Note that this consequence is true only if $N\ge 3$. But $\EE\left[|\overline{Y_t}|\right]= \EE[R_t]$ and so $\overline{X}_{g^2}\ne 0$ a.s. By symmetry $\overline{Y}_{g^2}\ne 0$. Returning back to (\ref{opo}), we deduce that $d(\overline{X}_t,\overline{Y}_t)-\frac{N-2}{N}(|\overline{X}_t|+|\overline{Y}_t|)$ is a martingale which proves Theorem \ref{main} (i).

Note that $g^1=g^X_t$, $g^2=g^Y_t$ and for $Z$ a WBM, the sets of zeros of $Z$ and $\overline{Z}$ are equal. Consequently $X_{g^Y_t}\ne 0$ and $Y_{g^X_t}\ne 0$ a.s. In particular $g^X_t\ne g^Y_t$ a.s and since $\{X_t=Y_t\}\subset\{g^X_t=g^Y_t\}$ (as $X$, $Y$ follow the same Brownian motion on the same ray), Theorem \ref{main} (ii) is also proved.
\begin{remark}
Using the convergence of $\EE[d(\overline{X^r_t},\overline{Y^r_t})]$ to $\EE[d(\overline{X_t},\overline{Y_t})]$, (\ref{yes}) and (\ref{mp}), we easily deduce that 
$$\int_0^t 1_{\{X_s\ne 0\}} dL_s(|Y|) =L_t(|Y|); \int_0^t 1_{\{Y_s\ne 0\}} dL_s(|X|)=L_t(|X|)$$
which is similar to Proposition \ref{pop} (ii). 
\end{remark}
Now it remains to prove the following
\begin{lemma}
We have $$\lim_{r\to 1} \EE[d(\overline{X^r_t},\overline{Y^r_t})] = \EE[d(\overline{X_t},\overline{Y_t})].$$
\end{lemma}
\begin{proof} From the convergence in law given in Lemma \ref{lem}, it is easily seen that $(\overline{X^r},\overline{Y^r})$ converges in law to $(\overline{X},\overline{Y})$. This is because $\overline{Z}$ is a continuous function of $Z$. Let $r_n$ be a sequence converging to $1$.
Skorokhod representation theorem says that it is possible to construct on some probability space $(\Omega',\mathcal A',\mathbb P')$, random variables $(X^n,Y^n)_{n\ge 1}$ and $(X^\infty,Y^\infty)$ such that 
for each $n$, $(X^n,Y^n)$ has the same law as $(\overline{X^{r_n}},\overline{Y^{r_n}})$ and $(X^{\infty},Y^\infty)$ has the same law as $(\overline{X},\overline{Y})$ and moreover $(X^n,Y^n)$ converges a.s. to $(X^{\infty},Y^\infty)$.  The lemma holds as soon as we prove $$\lim_{n\to \infty} \EE[d(X^n_t,Y^n_t)] = \EE[d(X^{\infty}_t,Y^\infty_t)].$$ 
For each $\epsilon>0$,
\begin{eqnarray*}
\EE[d(X^n_t,X^\infty_t)]&\le& \epsilon +  \EE[d(X^n_t,X^\infty_t)1_{\{d(X^n_t,X^\infty_t)>\epsilon\}}]\\
&\le& \epsilon + \EE[d(X^n_t,X^\infty_t)^2]^{1/2} \PP[d(X^n_t,X^\infty_t)>\epsilon]^{1/2}\\
&\le& \epsilon + C\times \PP[d(X^n_t,X^\infty_t)>\epsilon]^{1/2}
\end{eqnarray*}
for some finite constant $C$. Thus, $\limsup_n \EE[d(X^n_t,X^\infty_t)]=0$ and similarly \\
$\limsup_n \EE[d(Y^n_t,Y^\infty_t)]=0$. The lemma follows now using the triangle inequality.
\end{proof}

Let us now prove Theorem \ref{main} (iii). 

Denote by $\mathcal G$ the natural filtration of the Wiener coupling $(X,Y)$. For a random time $R$, let us recall the following $\sigma$-fields (see \cite{MR1655299} on page 286)
\begin{eqnarray}
\mathcal G_{R}&=&\sigma(U_R : U \ \text{is a}\ \mathcal G-\ \text{optional process}),\nonumber\\
\mathcal G_{R+}&=&\sigma(U_R : U \ \text{is a}\ \mathcal G-\ \text{progressive process}).\nonumber\
\end{eqnarray}
In the sequel, we will always consider the completions of these sigma-fields by null sets. Let $g^1=g^{X}_t, g^2=g^{Y}_t$. It is known (see for example Proposition 19 in \cite{MR1655299}), that $\varepsilon(X_t)$ is independent of $\mathcal G_{g^1}$ and $\varepsilon(X_t)$ is $\mathcal G_{g^1+}$ measurable (the same holds for $Y$). The event $\{g^1<g^2\}\in \mathcal G_{g^2}$ (see Proposition 13 in \cite{MR1655299}) and on this event, $\varepsilon(X_t)=\limsup_{\epsilon\rightarrow0+} \varepsilon(X_{(g^1+\epsilon)\wedge g^2})$. Since $(g^1+\epsilon)\wedge g^2\le g^2$, by Proposition 13 in \cite{MR1655299} again, $\mathcal G_{(g^1+\epsilon)\wedge g^2}\subset \mathcal G_{g^2}$ and so $\limsup_{\epsilon\rightarrow0+} \varepsilon(X_{(g^1+\epsilon)\wedge g^2})$ is $\mathcal G_{g^2}$-measurable. Take $f$ an indicator function on a subset of $\{1,\cdots,N\}$. By conditioning with respect to $\mathcal G_{g^2}$, we deduce 
$$\EE[f(\varepsilon(X_t)) f(\varepsilon(Y_t))1_{\{g^{1}<g^{2}\}}]=\EE[f(\varepsilon(Y_{t})] \EE[f(\varepsilon(X_t))1_{\{g^{1}<g^{2}\}}]$$
and 
$$\EE[f(\varepsilon(X_t)) f(\varepsilon(Y_t))1_{\{g^{2}<g^{1}\}}]=\EE[f(\varepsilon(X_{t})] \EE[f(\varepsilon(Y_t))1_{\{g^{2}<g^{1}\}}]$$
Summing and using $\mathbb P(g^1=g^2)=0$, we get
$$\EE[f(\varepsilon(X_t)) f(\varepsilon(Y_t))] = \EE[f(\varepsilon(X_{t})] \big(\EE[f(\varepsilon(X_t))1_{\{g^{1}<g^{2}\}}] + \EE[f(\varepsilon(Y_t))1_{\{g^{2}<g^{1}\}}]\big)$$
But $\{g^{1}<g^{2}\}=\{g^{2}<g^{1}\}^c$ a.s. Since $\mathcal G_{g^1}$ is complete, $\{g^{1}<g^{2}\}\in \mathcal G_{g^1}$ which is independent of $\varepsilon(X_t)$ so that 
$$\EE[f(\varepsilon(X_t))1_{\{g^{1}<g^{2}\}}]=\frac{1}{2} \EE[f(\varepsilon(X_t))].$$
Using the symmetry, we arrive at $\EE[f(\varepsilon(X_t)) f(\varepsilon(Y_t))]=\EE[f(\varepsilon(X_t))] \EE[f(\varepsilon(Y_t))]$.\\

\section{Interpretation using stochastic flows}\label{ff}
\leavevmode\par
\vspace{0.2cm}
This section gives an interpretation of the Wiener coupling using the Wiener stochastic flow of kernels solving the generalized interface equation considered in \cite{MR2905755}. For basic definitions of stochastic flows of mappings, kernels and real white noises, the reader is referred to \cite{MR2060298}.

For a family of doubly indexed random variables $Z=(Z_{s,t})_{s\le t}$, define $\mathcal F^Z_{s,t}=\sigma(Z_{u,v}, s\leq u\leq v\leq t)$ for all $s\le t$. The extension to flows of kernels of the interface SDE is the following.
\begin{definition}\label{df}
Let $K$ be a stochastic flow of kernels on $G$ and $\mathcal W=(\mathcal W^i, 1\le i\le N)$ be a family of independent real white noises. We say that $(K,\mathcal W)$ solves $(I)$ if
for all $s\leq t$, $f\in \mathcal D$ and $x\in G$, a.s.
$$K_{s,t}f(x)=f(x)+\sum_{i=1}^{N}\int_s^t
K_{s,u}(1_{E_i}f')(x)d\mathcal W^i_{s,u} + \frac{1}{2}\int_s^tK_{s,u}f''(x)du.$$
We say $K$ is a Wiener solution if for all $s\le t$, $\mathcal F^{K}_{s,t}\subset\mathcal F^{\mathcal W}_{s,t}$.
When $K$ is induced by a stochastic flow of mappings $\varphi$ ($K=\delta_{\varphi}$), we say $(\varphi,\mathcal W)$ is a solution of $(I)$.
\end{definition}
Note that when $K=\delta_{\varphi}$, the flow $\varphi$ defines a system of solutions to the interface SDE (\ref{def}) for all possible time and position initial conditions.

If $(K,\mathcal W)$ solves $(I)$, then $\mathcal F^{\mathcal W}_{s,t}\subset \mathcal F^{K}_{s,t}$ for all $s\le t$ \cite{MR2905755}. Therefore Wiener solutions are characterized by $\mathcal F^{\mathcal W}_{s,t}=\mathcal F^{K}_{s,t}$ for all $s\le t$. 

It has been proved in \cite{MR2905755} that there exists a law unique stochastic flow of mappings $\varphi$ and a real white noise $\mathcal W$ such that $(\varphi,\mathcal W)$ solves $(I)$. Filtering this flow with respect to $\mathcal W$ gives rise to a Wiener stochastic flow of kernels $K_{s,t}(x)=\EE[\delta_{\varphi_{s,t}(x)}|\mathcal F^{\mathcal W}_{s,t}]$ solution of $(I)$ which is unique up to modification.

In the case $N=2$ the Wiener flow and the flow of mappings coincide ($K=\delta_{\varphi}$) while $K\ne \delta_{\varphi}$ if $N\ge 3$ and other flows solving $(I)$ may exist \cite{MR2905755}. 

Let $(K,\mathcal W)$ be the Wiener stochastic flow which solves $(I)$. Then
$$Q_t(f\otimes g\otimes h)(x,y,w)=\EE[K_{0,t}f(x) K_{0,t}g(y)h(w+\mathcal W_{0,t})]$$
defines a Feller semigroup on $G^2\times \mathbb R^N$. Denote by $(X,Y,W)$ the Markov process associated to $(Q_t)_t$ and started from $(x,y,0)$.  
\begin{proposition}\label{hh}
 $(X,Y,W)$ is the Wiener coupling solution of $(I)$ with $X_0=x$ and $Y_0=y$. 
\end{proposition}
\begin{proof}
Note that $$\widetilde{Q}_t(f\otimes h)(x,w):=Q_t(f\otimes I\otimes h)(x,w)=\EE[f(\varphi_{0,t}(x)) h(w+\mathcal W_{0,t})]$$
In particular $(X,W)$ has the same law as $(\varphi_{0,t}(x),\mathcal W_{0,t})_{t\ge 0}$ and so it is a solution to $(I)$. The same holds for $(Y,W)$. Now it remains to prove that $X$ and $Y$ are independent given $W$. We will check that
$$\EE\left[\prod_{i=1}^n f_i(X_{t_i}) g_i(Y_{t_i}) h_i(W_{t_i})\right]=\EE\left[\prod_{i=1}^n \EE[f_i(X_{t_i})|W] \EE[g_i(Y_{t_i})|W] h_i(W_{t_i})\right]$$
for all measurable and bounded test functions $(f_i,g_i,h_i)_i$. Since $K$ is a measurable function of $\mathcal W$, we may assume $K$ (and so $\mathcal W$) is defined on the same space as $X$ and $Y$ and that $W_t=\mathcal W_{0,t}$. By an easy induction (see the proof of Proposition 4.1 in \cite{MROP}),
\begin{equation}\label{tyr}
\EE\left[\prod_{i=1}^n f_i(X_{t_i}) g_i(Y_{t_i}) h_i(W_{t_i})\right]=\EE\left[\prod_{i=1}^n K_{0,t_i}f_i(x) K_{0,t_i}g_i(y) h_i(W_{t_i})\right]
\end{equation}
From (\ref{tyr}), we also deduce $K_{0,t_i}f_i(x)=\EE[f_i(X_{t_i})|\mathcal F^W_{0,t_i}]$ and $K_{0,t_i}g_i(y)=\EE[g_i(Y_{t_i})|\mathcal F^W_{0,t_i}]$. This completes the proof. 
\end{proof}
Let $(\mathcal{W}_{s,t})_{s\le t}$ and $(\hat{\mathcal{W}}_{s,t})_{s\le t}$ be two independent real white noises and set $\mathcal{W}^r_{s,t}=r \mathcal{W}_{s,t} + \sqrt{1-r^2} \hat{\mathcal{W}}_{s,t}$. Denote by $K$ and $K^r$ the Wiener flows solutions of $(I)$ respectively driven by $\mathcal{W}$ and $\mathcal{W}^r$ and define
\begin{equation}\label{revi}
Q^r_t(f\otimes g\otimes h)(x,y,w)=\EE[K^r_{0,t}f(x) K_{0,t}g(y)h(w+\mathcal W_{0,t})]
\end{equation}
Then $Q^r$ is a Feller semigroup. Following the proof of Proposition \ref{hh}, one can prove that $(X^r,Y^r,W)$ given in Lemma \ref{lem} is the Markov process associated to $Q^r$ and starting from $(0,0,0)$. In particular this is also a Feller process.

\leavevmode\par
\vspace{0.2cm}


\textbf{Final remarks and open problems.}\\ 
There are interesting open problems related to the interface SDE. Let us mention some of them.
\begin{itemize}
\item What is the conditional law of $|X_t|$ (and more generally of $X_t$) given $W$?
\item What are the couplings which  ``interpolate'' between the coalescing coupling and the Wiener one? 
\item What are the stochastic flows which  ``interpolate'' between the coalescing flow and the Wiener one? (see \cite{MR2905755} for more details).

\end{itemize}
Let us finish with the following remark regarding the first question. Let $W$ be a standard Brownian motion and let $X^1,X^2,\cdots$ be WBMs started from $0$ such that $(X^i,W)$ is solution to $(I)$ with $X^i_0=0$ for all $i$ and $X^1,X^2,\cdots$ are independent given $W$. Then by the law of the large numbers for all $f\in C_0(G)$, a.s $\EE[f(X^1_t)|W]=\lim_n \frac{1}{n} \sum_{i=1}^{n} f(X^i_t)$ (see Section 2.6 in \cite{MR2060298}). 
\\
\\
\textbf{Acknowledgments.}  We thank the reviewer for his/her thorough review and highly appreciate the comments and suggestions which significantly improved two versions of the paper. In particular the reviewer suggested the present construction of the perturbation process instead of a stochastic flows based construction given in the first version which used the semi group $Q^r$ (\ref{revi}).

\bibliographystyle{plain}
\bibliography{Bil4}

\end{document}